\documentclass[a4paper,reqno,11pt,oneside]{amsart}

\usepackage[all]{xy}
\xyoption{poly}
\usepackage{amssymb}
\usepackage{amstext}
\usepackage{amsmath}
\usepackage{amscd}
\usepackage{amsthm}
\usepackage{amsfonts}
\usepackage{latexsym}
\usepackage{graphicx}
\usepackage{color}
\usepackage{geometry}
\makeatletter
  
  \@addtoreset{equation}{section}
\makeatother

\makeatletter
\@namedef{subjclassname@2020}{%
  \textup{2020} Mathematics Subject Classification}
\makeatother

\newcommand{\calE}{\mathcal{E}}

\newcommand{\calV}{\mathcal{V}}

\newcommand{\ZZ}{\mathbb{Z}}

\newcommand{\FF}{\mathbb{F}}
\newcommand{\PP}{\mathbb{P}}

\newcommand{\Hom}{\operatorname{Hom}}
\newcommand{\Ext}{\operatorname{Ext}}

\def\opn#1#2{\def#1{\operatorname{#2}}} 
\opn\rank{rank} \opn\mnull{null} \opn\Iso{Iso}

\def\sfA{\mathsf{A}}
\def\sfC{\mathsf{C}}
\def\coh{\operatorname{\mathsf{coh}}}
\def\CM{\mathsf{CM}}
\def\uCM{\underline{\mathsf{CM}}}
\def\Db{\mathsf{D^b}}
\def\mod{\operatorname{\mathsf{mod}}}
\def\e{\varepsilon}
\opn\Proj{Proj}

\def\rnum#1{\expandafter{\romannumeral #1}}
\def\Rnum#1{\uppercase\expandafter{\romannumeral #1}}

\newtheorem{thm}{Theorem}[section]

\newtheorem{lem}[thm]{Lemma}

\theoremstyle{definition}
\newtheorem{defi}[thm]{Definition}
\newtheorem{ex}[thm]{Example}
\newtheorem{nota}[thm]{Notation}

\theoremstyle{remark}
\newtheorem{rem}[thm]{Remark}

\begin{document}

\title [Combinatorial study of stable categories]
{Combinatorial study of stable categories of graded Cohen--Macaulay modules over skew quadric hypersurfaces}

\author{Akihiro Higashitani}
\address{Department of Pure and Applied Mathematics, 
Graduate School of Information Science and Technology, 
Osaka University, 
1-5, Yamadaoka, Suita, Osaka 565-0871, Japan}
\email{higashitani@ist.osaka-u.ac.jp}

\author{Kenta Ueyama}
\address{Department of Mathematics,
Faculty of Education,
Hirosaki University,
1 Bunkyocho, Hirosaki, Aomori 036-8560, Japan}
\email{k-ueyama@hirosaki-u.ac.jp}



\begin{abstract}
In this paper, we present a new connection between representation theory of noncommutative hypersurfaces and combinatorics.
Let $S$ be a graded ($\pm 1$)-skew polynomial algebra in $n$ variables of degree $1$ and $f =x_1^2 + \cdots +x_n^2 \in S$.
We prove that the stable category $\mathsf{\underline{CM}}^{\mathbb Z}(S/(f))$ of graded maximal Cohen--Macaulay module over $S/(f)$ can be completely computed using the four graphical operations.
As a consequence, $\mathsf{\underline{CM}}^{\mathbb Z}(S/(f))$ is equivalent to the derived category
$\mathsf{D^b}(\operatorname{\mathsf{mod}} k^{2^r})$, and this $r$ is obtained as the nullity of a certain matrix over ${\mathbb F}_2$.
Using the properties of Stanley--Reisner ideals,
we also show that the number of irreducible components of the point scheme of $S$ that are isomorphic to ${\mathbb P}^1$
is less than or equal to $\binom{r+1}{2}$.
\end{abstract}

\maketitle

\section{Introduction}

Triangulated categories play an increasingly important role
in many areas of mathematics, including representation theory, (commutative and noncommutative) algebraic geometry,
algebraic topology, and mathematical physics.
In particular, there are two major classes of triangulated categories, namely,
the (bounded) derived categories $\Db(\sfA)$ of abelian categories $\sfA$ and
the stable categories $\underline{\sfC}$ of Frobenius categories $\sfC$.
For example, the derived categories $\Db(\coh X)$ of coherent sheaves on algebraic varieties $X$ have been studied extensively in algebraic geometry, and the stable categories $\uCM(A)$ of maximal Cohen--Macaulay modules over (not necessary commutative) Gorenstein algebras $A$
have been studied extensively in representation theory of algebras.
In this paper, we compute the stable categories $\uCM^{\ZZ}(A)$ of graded maximal Cohen--Macaulay modules over
certain noncommutative quadric hypersurface rings $A$ (in the sense of Smith and Van~den~Bergh \cite{SV})
using combinatorial methods. 

Throughout let $k$ be an algebraically closed field of characteristic not $2$.
It is well-known that if $A$ is the homogeneous coordinate ring of a smooth quadric hypersurface in $\PP^{n-1}$,
then $A $ is isomorphic to $ k[x_1,\dots, x_n]/(x_1^2+\cdots +x_n^2)$, so we have
an equivalence of triangulated categories
\begin{align} \label{eq.c}
\uCM^{\ZZ}(A)\cong
\begin{cases}
\uCM^{\ZZ}(k[x_1]/(x_1^2)) \cong \Db(\mod k)  &\text { if $n$ is odd,} \\
\uCM^{\ZZ}(k[x_1, x_2]/(x_1^2+x_2^2)) \cong \Db(\mod k^2) & \text { if $n$ is even}
\end{cases}
\end{align}
by Kn\"orrer's periodicity theorem (\cite[Theorem 3.1]{Kn}).
The main aim of this paper is to give a skew generalization of this equivalence.
More precisely, we consider the following setting.

\begin{nota}\label{nota}
For a symmetric matrix $\e:= (\e_{ij}) \in M_n(k)$ such that $\e_{ii}=1$ and $\e_{ij}=\e_{ji}=\pm 1$, we fix the following notations: 
\begin{enumerate}
\item{} the standard graded algebra $S_{\e}:=k\langle x_1, \dots, x_n\rangle /(x_ix_j-\e_{ij}x_jx_i)$, called a \emph{$(\pm 1)$-skew polynomial algebra} in $n$ variables, 
\item{} the point scheme $\calE_{\e}$ of $S_{\e}$, 
\item{} the central element $f_{\e}:=x_1^2+\cdots +x_n^2\in S_{\e}$, 
\item{} $A_{\e}:=S_{\e}/(f_{\e})$, and 
\item{} the graph $G_{\e}$ where 
$V(G_{\e})=\{1, \dots , n\}$ and $E(G_{\e})=\{\{i,j\} \mid \e_{ij}=\e_{ji}=1, i \neq j \}$.  
\end{enumerate} 
\end{nota}

In \cite{Ukp}, the second author gave a classification theorem for $\uCM^{\ZZ}(A_{\e})$ with $n\leq 5$.
After that, in \cite{MU}, Mori and the second author introduced graphical methods to compute $\uCM^{\ZZ}(A_{\e})$.
They presented the four operations for $G_\e$, called
mutation, relative mutation, Kn\"orrer reduction, and two points reduction,
and showed that $\uCM^{\ZZ}(A_{\e})$  can be completely computed up to $n\leq 6$ by using these four graphical operations
(see \cite[Section 6.4]{MU}).
We first extend this result to arbitrary $n \in \ZZ_{>0}$.

\begin{thm}\label{thm1}
Let $A_\e$ and $G_\e$ be as in Notation \ref{nota}.
By using mutation, relative mutation, Kn\"orrer reduction, and two points reduction finitely many times,
$G_\e$ can be reduced to the one-vertex graph, and there exists an equivalence of triangulated categories
\[\uCM^{\ZZ}(A_\e) \cong \Db(\mod k^{2^r})\]
where $r$ is the number of times we applied two points reduction.
\end{thm}

Thus, we can completely compute $\uCM^{\ZZ}(A_\e)$ by purely combinatorial methods.
Thanks to this theorem, we obtain the following two consequences.

\begin{thm} \label{thm2}
Let $A_\e$ and $G_\e$ be as in Notation \ref{nota}.
Then we have an equivalence of triangulated categories
\[\uCM^{\ZZ}(A_\e) \cong \Db(\mod k^{2^r})\]
where
\[r= \mnull_{\FF_2}
\begin{pmatrix}  
 & & &1 \\
 &M(G_\e) & &\vdots \\
 & & &1 \\
1 &\cdots &1 &0
\end{pmatrix}
\]
and $M(G_\e)$ is the adjacency matrix of $G_\e$ over $\FF_2$.
In particular, $A_\e$ has $2^{r}$ indecomposable non-projective
graded maximal Cohen--Macaulay modules up to isomorphism and degree shifts.
\end{thm}

\begin{thm} \label{thm3}
Let $A_\e$ be as in Notation \ref{nota}.
Then $A_\e$ is a noncommutative graded isolated singularity (in the sense of \cite{Uis}).
\end{thm}

It is easy to see that Theorem~\ref{thm2} is a generalization of (\ref{eq.c}).
Moreover, Theorem~\ref{thm3} tells us that $A_\e$ is a homogeneous coordinate ring of a noncommutative ``smooth'' quadric hupersurface
(see \cite{SV}, \cite{MU} for details).

Let $A$ be a graded algebra finitely generated in degree $1$.
A graded $A$-module $M$ is called a \emph{point module} if $M$ is cyclic and has Hilbert series $H_M(t)=(1-t)^{-1}$.
If $A$ is commutative, these modules correspond to the closed points of the projective scheme $\Proj A$.
In \cite{ATV}, Artin, Tate, and Van den Bergh introduced a scheme $\calE$ whose closed points parameterize the isomorphism classes of point modules over $A$;
so it is called the \emph{point scheme} of $A$.
Since then, point schemes are an essential tool to study graded algebras in noncommutative algebraic geometry.

In \cite[Conjecture 1.3]{Ukp}, it was conjectured that the structure of $\uCM^{\ZZ}(A_\e)$ is determined by
the number of irreducible components of the point scheme $\calE_\e$ of $S_\e$ that are isomorphic to $\PP^1$.
This is true if $n\leq 6$ (see \cite[Theorem 6.20]{MU}),
but unfortunately, it is known to fail for $n=7$ (see \cite[Remark 6.21]{MU}).
Using a similar approach to the proof of Theorem~\ref{thm1} and the point of view of Stanley--Reisner ideals,
we give a combinatorial proof of the following result.

\begin{thm}\label{thm4}
Let $A_\e$ be as in Notation \ref{nota} so that there is a non-negative integer $r$ such that $\uCM^{\ZZ}(A_\e) \cong \Db(\mod k^{2^r})$.
Then the number $\ell_\e$ of irreducible components of $\calE_\e$ that are isomorphic to $\PP^1$ is less than or equal to $\binom{r+1}{2}$.
\end{thm}

This theorem shows that the upper bound of $\ell_\e$ which appears in \cite[Conjecture 1.3]{Ukp} holds true for arbitrary $n \in \ZZ_{>0}$.

In short, the results and discussion in this paper establish a novel connection between representation theory of noncommutative hypersurfaces and combinatorics.

\section{Graphical Methods to Compute Stable Categories}
\subsection{Stable Categories of Graded Maximal Cohen--Macaulay Modules}
Throughout this paper, we continue to use Notation \ref{nota}.
It is well-known that $S_\e=k\langle x_1, \dots, x_n\rangle /(x_ix_j-\e_{ij}x_jx_i)$ is a noetherian Koszul AS-regular algebra.
Since $f_\e=x_1^2+\cdots +x_n^2$ is a regular central element of $S_\e$ of degree $2$,
it follows that $A_\e=S_\e/(f_\e)$ is a noetherian Koszul AS-Gorenstein algebra.
A finitely generated graded $A_\e$-module $M$ is called \emph{maximal Cohen--Macaulay} if $\Ext^i_{A_\e}(M, A_\e)=0$ for all $i \neq 0$.
We denote by $\CM^{\ZZ}(A_\e)$ the category of (finitely generated) graded maximal Cohen--Macaulay $A_\e$-modules
with degree preserving $A_\e$-module homomorphisms.

The stable category of graded maximal Cohen--Macaulay modules, denoted by $\uCM^{\ZZ}(A_\e)$, has the same objects as $\CM^{\ZZ}(A_\e)$
and the morphism space is given by
\[ \Hom_{\uCM^{\ZZ}(A_\e)}(M, N) = \Hom_{\CM^{\ZZ}(A_\e)}(M,N)/P(M,N) \]
where $P(M,N)$ consists of degree preserving $A_\e$-module homomorphisms factoring through a graded projective module.
Since $A_\e$ is AS-Gorenstein, 
$\CM^{\ZZ}(A_\e)$ is a Frobenius category and
$\uCM^{\ZZ}(A_\e)$ is a triangulated category whose translation functor $[1]$ is given by 
the cosyzygy functor $\Omega^{-1}$ (see \cite[Section 4]{Bu}, \cite[Theorem 3.1]{SV}).

\subsection{Two Mutations and Two Reductions of Graphs} 

A \emph{graph} $G$ consists of a set of vertices $V(G)$ and a set of edges $E(G)$ between two vertices. 
In this paper, we always assume that $V(G)$ is a finite set and $G$ has neither loops nor multiple edges. 
An edge between two vertices $v, w\in V(G)$ is written by $vw \in E(G)$. 

Let $G$ be a graph. 
A graph $G'$ is the \emph{induced subgraph} of $G$ induced by $V' \subset V(G)$ if $vw \in E(G')$ whenever $v,w \in V'$ and $vw \in E(G)$. 
For a subset $W \subset V(G)$, we denote by $G \setminus W$ the induced subgraph of $G$ induced by $V(G) \setminus W$. 
For a vertex $v \in V(G)$, let $N_G(v)=\{u \in V(G) \mid uv \in E(G)\}$. 

First, we recall the notion of two mutations, which preserve the stable category of
graded maximal Cohen--Macaulay modules.

\begin{defi} [Mutation, {\cite[Definition 6.3]{MU}}]
Let $G$ be a graph and $v\in V(G)$. 
The \emph{mutation} $\mu_v(G)$ of $G$ at $v$ is the graph $\mu_v(G)$ where $V(\mu_v(G))=V(G)$ and 
$$E(\mu_v(G))=\{ vw \mid w \in V(G) \setminus N_G(v) \} \cup E(G \setminus \{v\}).$$
\end{defi}

\begin{rem}
The notion of mutation $\mu_v(G)$ is called \emph{switching} on $\{v\}$ in the context of algebraic graph theory. See \cite[Section 11.5]{GR}. 
Moreover, we see that applying the consecutive mutations $\mu_{v_1},\ldots,\mu_{v_m}$ for distinct vertices $v_1,\ldots,v_m$ corresponds to switching on $\{v_1,\ldots,v_m\}$. 
In particular, the resulting graph is independent of the choice of the order of consecutive mutations. 
\end{rem}

\begin{ex}\label{ex.m}
\begin{enumerate}
\item
\[G= \xy /r2pc/: 
{\xypolygon5{~={90}~*{\xypolynode}~>{}}},
"1";"2"**@{-},
"1";"3"**@{-},
"1";"5"**@{-},
"2";"4"**@{-},
"3";"4"**@{-},
"4";"5"**@{-}
\endxy
\quad \Longrightarrow  \quad
\mu_{1}(G) = \xy /r2pc/: 
{\xypolygon5{~={90}~*{\xypolynode}~>{}}},
"1";"4"**@{-},
"2";"4"**@{-},
"3";"4"**@{-},
"4";"5"**@{-}
\endxy
\]
\item
\[G= \xy /r2pc/: 
{\xypolygon5{~={90}~*{\xypolynode}~>{}}},
"1";"3"**@{-},
"3";"5"**@{-},
"5";"2"**@{-},
"2";"4"**@{-},
"4";"1"**@{-},
"3";"4"**@{-}
\endxy
\quad \Longrightarrow  \quad
\mu_{4}(\mu_{3}(G)) = \xy /r2pc/: 
{\xypolygon5{~={90}~*{\xypolynode}~>{}}},
"2";"3"**@{-},
"3";"4"**@{-},
"4";"5"**@{-},
"5";"2"**@{-},
\endxy
\]
\end{enumerate}
\end{ex}

\begin{lem} [Mutation Lemma, {\cite[Lemma 6.5]{MU}}] \label{lem.mg} \label{lem.ramu} 
If $G_{\e'}=\mu_v(G_{\e})$ for some $v \in V(G_{\e})$, then $\uCM^{\ZZ}(A_{\e})\cong \uCM^{\ZZ}(A_{\e'})$.   
\end{lem} 

\begin{defi}[Relative Mutation, {\cite[Definition 6.6]{MU}}]
Let $u,v \in V(G)$ be distinct vertices.
Then the \emph{relative mutation} $\mu_{v \leftarrow u}(G)$ of $G$ at $v$ with respect to $u$ is the graph $\mu_{v \leftarrow u}(G)$
where $V(\mu_{v \leftarrow u}(G))=V(G)$ and
\begin{align*}
E(\mu_{v \leftarrow u}(G))
=\{vw \mid w \in N_G(u) \setminus N_G(v)\} \cup \{vw \mid w \in N_G(v) \setminus N_G(u)\} \cup E(G \setminus \{v\}).
\end{align*}
\end{defi}

\begin{ex}\label{ex.rm}
\begin{enumerate}
\item
\[G= \xy /r2pc/: 
{\xypolygon6{~={90}~*{\xypolynode}~>{}}},
"1";"3"**@{-},
"1";"5"**@{-},
"2";"3"**@{-},
"2";"4"**@{-},
"3";"5"**@{-},
"4";"5"**@{-}
\endxy
\quad \Longrightarrow  \quad
\mu_{1 \leftarrow 2}(G) = \xy /r2pc/: 
{\xypolygon6{~={90}~*{\xypolynode}~>{}}},
"1";"4"**@{-},
"1";"5"**@{-},
"2";"3"**@{-},
"2";"4"**@{-},
"3";"5"**@{-},
"4";"5"**@{-}
\endxy
\]
\item
\[G= \xy /r2pc/: 
{\xypolygon6{~={90}~*{\xypolynode}~>{}}},
"2";"3"**@{-},
"3";"4"**@{-},
"4";"2"**@{-},
"2";"5"**@{-},
"5";"6"**@{-},
"6";"2"**@{-}
\endxy
\quad \Longrightarrow  \quad
\mu_{2 \leftarrow 4}(\mu_{2 \leftarrow 3}(G)) = \xy /r2pc/: 
{\xypolygon6{~={90}~*{\xypolynode}~>{}}},
"3";"4"**@{-},
"2";"5"**@{-},
"5";"6"**@{-},
"6";"2"**@{-}
\endxy
\]
\end{enumerate}
\end{ex}

\begin{lem}[Relative Mutation Lemma, {\cite[Lemma 6.7]{MU}}]\label{lem.M2}
Suppose that $G_{\e}$ contains an isolated vertex $u$. 
If $G_{\e'}=\mu_{v \leftarrow w}(G_\e)$ for some distinct vertices $v, w \in V(G_\e)$ not equal to $u$, 
then $\uCM^{\ZZ}(A_{\e})\cong \uCM^{\ZZ}(A_{\e'})$.
\end{lem} 

Next, we recall two ways to reduce the number of variables in computing the stable category of
graded maximal Cohen--Macaulay modules over $A_\e$.

\begin{defi}
An \emph{isolated edge} $vw$ of a graph $G$ is an edge $vw \in E(G)$ such that $N_G(v)=\{w\}$ and $N_G(w)=\{v\}$. 
\end{defi} 

\begin{lem} [Kn\"orrer's Reduction, {\cite[Lemma 6.17]{MU}}] \label{lem.R2} 
Suppose that $vw$ is an isolated edge in $G_{\e}$. If $G_{\e'}=G_{\e}\setminus \{v, w\}$, 
then 
$\uCM^{\ZZ}(A_{\e})\cong \uCM^{\ZZ}(A_{\e'})$. 
\end{lem} 

\begin{lem} [Two Points Reduction, {\cite[Lemma 6.18]{MU}}] \label{lem.R1} 
Suppose that $v, w \in V(G_{\e})$ are two distinct isolated vertices.
If $G_{\e'}=G_{\e}\setminus \{v\}$, then $\uCM^{\ZZ}(A_{\e}) \cong \uCM^{\ZZ}(A_{\e'}) \times \uCM^{\ZZ}(A_{\e'})$.
\end{lem}

\section{Proofs of Theorems \ref{thm1}, \ref{thm2}, and \ref{thm3}}

In this section, we present proofs of Theorems \ref{thm1}, \ref{thm2}, and \ref{thm3}.

\begin{lem}\label{lem.mu}
Let $G$ be a graph and $v$ its vertex. Then there exists a sequence of mutations $\mu_{v_1},\ldots,\mu_{v_m}$ such that 
$v$ becomes an isolated vertex in $\mu_{v_m}(\mu_{v_{m-1}}( \cdots \mu_{v_1}(G) \cdots ))$. 
\end{lem}
\begin{proof}
We may apply the mutations at all $u$'s in $N_G(v)$. (See Example~\ref{ex.m} (2) for an example.) 
\end{proof}

Given two non-negative integers $a$ and $b$, let $G(a,b)$ denote the graph on the set of vertices $\{u_i,u_i' \mid i=1,\ldots,a\} \cup \{u_j'' \mid j = 1,\ldots,b\}$ 
with its set of edges $\{u_iu_i' \mid i=1,\ldots,a\}$. Namely, $G(a,b)$ consists of $a$ isolated edges and $b$ isolated vertices. 

\begin{lem}\label{lem.rm}
Let $G$ be a graph with $n$ vertices having at least one isolated vertex. 
Then there exists a sequence of relative mutations $\mu_{v_1 \leftarrow w_1}, \ldots, \mu_{v_k \leftarrow w_k}$ such that 
$$\mu_{v_k \leftarrow w_k}(\mu_{v_{k-1} \leftarrow w_{k-1}}( \cdots \mu_{v_1 \leftarrow w_1}(G) \cdots )) = G(\alpha,\beta),$$ 
where $2\alpha+\beta=n$ and $\beta >0$. 
\end{lem}
\begin{proof}
We prove the statement by induction on $n$. The statement is trivial in the case $n=1$ since $G$ is already equal to $G(0,1)$. 

Suppose that $n>1$. Take an isolated vertex $v_0$ of $G$. We fix a vertex $v$ in $G$ with $v \neq v_0$ and we let $G' = G \setminus \{v\}$. 
By the hypothesis of induction, there exists a sequence of relative mutations on $G'$ 
such that $G'$ can be transformed into $G(\alpha',\beta')$ for some $\alpha' \in \ZZ_{\geq 0}$ and $\beta' \in \ZZ_{>0}$. 
Let $\widetilde{G}$ be the graph after applying those relative mutations to $G$. 
Then $\widetilde{G} \setminus \{v\} = G(\alpha',\beta')$. Note that $v_0$ is still an isolated vertex in $\widetilde{G}$. 

Let $\{u_i,u_i' \mid i=1,\ldots,\alpha'\} \cup \{u_j'' \mid j = 1,\ldots,\beta'-1\} \cup \{v_0\}$ 
denote the set of vertices of $\widetilde{G} \setminus \{v\}$ 
and let $\{u_iu_i' \mid i=1,\ldots,\alpha'\}$ be the set of its edges. We let the set of edges of $\widetilde{G}$ as follows: 
\begin{align*}
&\{vu_i,u_iu_i' \mid i=1,\ldots,p\} \cup \{vu_i, vu_i', u_iu_i' \mid i=p+1,\ldots,q\} \cup \{u_iu_i' \mid i=q+1,\ldots,\alpha'\} \\
&\cup \{vu_i'' \mid i=1,\ldots,r\}, 
\end{align*}
where $0 \leq p \leq q \leq \alpha'$ and $0 \leq r \leq \beta'-1$. Then, by applying 
\begin{align*}
\underbrace{\mu_{v \leftarrow u_1'},\ldots,\mu_{v \leftarrow u_p'}}_\text{breaks $vu_i$ for $i=1,\ldots,p$}, \quad
\underbrace{\mu_{v \leftarrow u_{p+1}}, \mu_{v \leftarrow u_{p+1}'},\ldots,\mu_{v \leftarrow u_q}, \mu_{v \leftarrow u_q'}}_\text{breaks $vu_i$ and $vu_i'$ for $i=p+1,\ldots,q$}, \quad
\underbrace{\mu_{u_2'' \leftarrow u_1''},\ldots,\mu_{u_r'' \leftarrow u_1''}}_\text{breaks $vu_j''$ for $j=2,\ldots,r$ if $r \geq 2$}, 
\end{align*}
$\widetilde{G}$ eventually becomes the graph whose edge set is 
$$\{u_iu_i' \mid i=1,\ldots,\alpha'\} \cup \{vu_1''\} \text{ if $r \geq 1$} \quad \text{or}\quad \{u_iu_i' \mid i=1,\ldots,\alpha'\}\text{ if $r=0$}.$$ 
This is nothing but $G(\alpha'+1,\beta'-1)$ if $r \geq 1$ (i.e. $\beta' \geq 2$) and $G(\alpha',\beta'+1)$ if $r=0$. 
See the figure below.

\begin{align*}&\xymatrix@R=0.8pc@C=0.5pc{
&&&&v \ar@{-}[ddllll] \ar@{-}[ddll] \ar@{-}[ddl] \ar@{-}[ddddl] \ar@{-}[ddr] \ar@{-}[ddddr] \ar@{-}@/^25pt/[dddrrrrrr] \ar@{-}@/^25pt/[dddrrrrrrrr]
&&&&&&&& v_0\\
\\
u_1 \ar@{-}[dd]& &u_p \ar@{-}[dd] &u_{p+1} \ar@{-}[dd]&&u_q \ar@{-}[dd] &u_{q+1} \ar@{-}[dd] & &u_{\alpha'} \ar@{-}[dd]\\
 &\cdots &&&\cdots &&&\cdots &&&u''_1 &\cdots &u''_r &u''_{r+1} &\cdots &u''_{\beta'-1}\\
u_1' & &u_p'&u_{p+1}'&&u_q'&u_{q+1}' &&u'_{\alpha'}
}\\ 
& \qquad\qquad\qquad\qquad\qquad \text{\rotatebox[origin=c]{270}{$\leadsto$} relative mutations}\\
&\xymatrix@R=0.8pc@C=0.5pc{
&&&&v \ar@{-}@/^25pt/[dddrrrrrr] &&&&&&&& v_0\\
\\
u_1 \ar@{-}[dd]& &u_p \ar@{-}[dd] &u_{p+1} \ar@{-}[dd]&&u_q \ar@{-}[dd] &u_{q+1} \ar@{-}[dd] & &u_{\alpha'} \ar@{-}[dd]\\
 &\cdots &&&\cdots &&&\cdots &&&u''_1 &\cdots &u''_r &u''_{r+1} &\cdots &u''_{\beta'-1}\\
u_1' & &u_p'&u_{p+1}'&&u_q'&u_{q+1}' &&u'_{\alpha'}\\
}
\end{align*}
\end{proof}

Using Lemmas~\ref{lem.mu}~and~\ref{lem.rm}, we can prove Theorem~\ref{thm1}.

\begin{proof}[\textit{\textbf{Proof of Theorem \ref{thm1}}}]
By Lemma~\ref{lem.mu}, we can transform $G_\e$ into a graph $G_{\e'}$ having at least one isolated vertex by using mutations several times.
Moreover, by Lemma~\ref{lem.rm}, it follows that $G_{\e'}$ can be transformed into $G_{\e''}:= G(\alpha, \beta)$ with $\alpha\geq 0, \beta>0$ 
by using relative mutations several times. 
It is easy to see that $G_{\e''}$ can be reduced to the one-vertex graph by applying Kn\"orrer's reductions $\alpha$ times and two points reductions $(\beta-1)$ times.
Hence we see that every $G_\e$ can be reduced to the one-vertex graph using two mutations and two reductions. In addition, we have 
\[ \uCM^{\ZZ}(A_{\e}) \cong \uCM^{\ZZ}(A_{\e'}) \cong \uCM^{\ZZ}(A_{\e''}) \cong \uCM^{\ZZ}(k[x]/(x^2))^{\times 2^{\beta-1}}
\cong \Db(\mod k)^{\times 2^{\beta-1}} \cong \Db(\mod k^{2^{\beta-1}}).\]
\end{proof}

For a matrix $M$ with its entry in $\FF_2$, let $\rank_{\FF_2}(M)$ (resp. $\mnull_{\FF_2}(M)$) denote the rank (resp. the nullity) of $M$ over $\FF_2$, which is called the \emph{binary rank} (resp. the \emph{binary nullity}).

For a graph $G$, let $M(G)$ denote the adjacency matrix of $G$. 
We denote by $X(G)$ the adjacency matrix of the graph whose vertex set is $V(G) \cup \{v'\}$ with its edge set $E(G) \cup \{vv' \mid v \in V(G)\}$. 
Note that $X(G)$ looks like as follows: $$X(G)=\begin{pmatrix} 
 & & &1 \\
 &M(G) & &\vdots \\
 & & &1 \\
1 &\cdots &1 &0
\end{pmatrix}.$$ 
In what follows, we will regard each entry of $M(G)$ and $X(G)$ as an element of $\FF_2$. 
\begin{lem}\label{lem.rank}
Work with the same situation and notation as in Lemma~\ref{lem.rm}. Then we have $$\rank_{\FF_2}(X(G))=2\alpha+2.$$
\end{lem}
\begin{proof}
By definition of mutation,  we can observe the following: 
\begin{align*}
(E_{i,n+1}+E)X(G)(E_{n+1,i}+E)&=\begin{pmatrix} 
 & & &1 \\
 &M(\mu_{i}(G)) & &\vdots \\
 & & &1 \\
1 &\cdots &1 &0
\end{pmatrix}, 
\end{align*}
where $E_{i,j}$ is the matrix such that $(i,j)$-entry is $1$ and the other entries are all $0$, and $E$ is the identity matrix. 
(Note that this holds true without assuming that $G$ has an isolated vertex.)
Since $G$ has an isolated vertex, we may assume that $n$-th row (resp. $n$-th column) of $X(G)$ is $\left(\begin{smallmatrix} 0 &\cdots &0 &1 \end{smallmatrix}\right)$ (resp. $\left(\begin{smallmatrix} 0 \\ \vdots \\0 \\1 \end{smallmatrix}\right)$).
Then, by definition of relative mutation, we can observe the following:
\begin{align*}
(E_{i,j}+E_{i,n}+E)X(G)(E_{j,i}+E_{n,i}+E)&=\begin{pmatrix} 
 & & &1 \\
 &M(\mu_{i \leftarrow j}(G)) & &\vdots \\
 & & &1 \\
1 &\cdots &1 &0
\end{pmatrix}. 
\end{align*}
By Lemmas~\ref{lem.mu} and \ref{lem.rm} together with the above observation, we see that there exists a sequence of regular matrices $P_1,\ldots,P_N,Q_1,\ldots,Q_N$ such that 
$$Q_N \cdots Q_1 X(G) P_1 \cdots P_N=
\begin{pmatrix}
 & & &1 \\
 &\underbrace{A \oplus \cdots \oplus A}_\alpha \oplus O & &\vdots \\
 & & &1 \\
1 &\cdots &1 &0 \end{pmatrix},$$
where $A=\begin{pmatrix}0 &1 \\ 1 &0 \end{pmatrix}$ and $O$ denotes the zero matrix of size ${n-2\alpha}$. 
We can easily see that $\rank_{\FF_2}(X(G))=\rank_{\FF_2}(Q_N \cdots Q_1 X(G) P_1 \cdots P_N)=2\alpha+2$, as required. 
\end{proof}

\begin{rem}
In \cite[Theorem 8.10.2]{GR}, an interpretation of the binary rank of the adjacency matrix of a graph $G$ 
was given in terms of a local operation for $G$, called a \emph{local complement}. 
A local complement seems to be related to a relative mutation, but it is not clear how they are related.
\end{rem}

Now we are ready to prove Theorems \ref{thm2} and \ref{thm3}.

\begin{proof}[\textit{\textbf{Proof of Theorem \ref{thm2}}}]
By the proof of Theorem~\ref{thm1}, it suffices to show that $\beta-1=\mnull_{\FF_2}(X(G_\e))$. 
Let $n$ be the number of vertices of $G_\e$. Then $\beta=n-2\alpha$. Therefore, 
\begin{align*}
\beta-1 &=n-2\alpha-1=(n+1)-(2\alpha+2)=n+1-\rank_{\FF_2}(X(G_\e)) \quad\text{(by Lemma~\ref{lem.rank})} \\
&=\mnull_{\FF_2}(X(G_\e)). 
\end{align*}
By the proof of \cite[Theorem 5.4]{MU}, it follows that $A_\e$ has $2^{\mnull_{\FF_2}(X(G_\e))}\ (=\#\mathrm{Ker}_{\FF_2}(X(G_\e)))$
indecomposable non-projective graded maximal Cohen--Macaulay modules up to isomorphism and degree shifts.
\end{proof}

\begin{proof}[\textit{\textbf{Proof of Theorem \ref{thm3}}}]
Since $A_\e$ is finite Cohen--Macaulay representation type by Theorem~\ref{thm2},
the result follows from \cite[Theorem 3.4]{Uis}.
\end{proof}

\begin{ex}
If
\[
G_\e=
\xy /r2pc/: 
{\xypolygon5{~={90}~*{\xypolynode}~>{}}},
"1";"2"**@{-},
"2";"3"**@{-},
"3";"4"**@{-},
"4";"1"**@{-},
"1";"5"**@{-}
\endxy,
\]
then
\[
\mnull_{\FF_2}
\begin{pmatrix}
0 &1 &0 &1 &1 &1 \\
1 &0 &1 &0 &0 &1 \\
0 &1 &0 &1 &0 &1 \\
1 &0 &1 &0 &0 &1 \\
1 &0 &0 &0 &0 &1 \\ 
1 &1 &1 &1 &1 &0
\end{pmatrix}
=2,
\]
so we have $\uCM^{\ZZ}(A_{\e}) \cong \Db(\mod k^4)$.

On the other hand, by applying the mutation $\mu_1$, the relative mutations $\mu_{2 \leftarrow 1}$ and $\mu_{4 \leftarrow 1}$, 
$G_\e$ becomes as follows: 
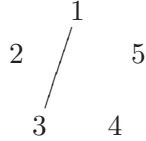
\[
\xy /r2pc/: 
{\xypolygon5{~={90}~*{\xypolynode}~>{}}},
"1";"3"**@{-}
\endxy
\]
Hence $\alpha=1$ and $\beta=3$ in the sense of Lemma~\ref{lem.rm}, i.e., $\mu_{4 \leftarrow 1}(\mu_{2 \leftarrow 1}(\mu_1(G_\e)))=G(1,3)$.
This also shows $\uCM^{\ZZ}(A_{\e}) \cong \Db(\mod k^4)$.
\end{ex}

\section{Proof of Theorem \ref{thm4}}

This section is devoted to the proof of Theorem~\ref{thm4}.

For a graph $G$, let $\Iso(G)$ be the set of isolated vertices of $G$ and let $i(G)=\#\Iso(G)$. 

For two graphs $G$ and $G'$, we write $G \sim G'$ if $G$ can be transformed into $G'$ by finitely many mutations and relative mutations. 
By Lemmas~\ref{lem.mu} and \ref{lem.rm}, we know that $G \sim G(\alpha,\beta)$ for some $\alpha \in \ZZ_{\geq 0}$ and $\beta \in \ZZ_{>0}$. 
Moreover, by Lemma~\ref{lem.rank}, we also see that $G(\alpha,\beta) \not\sim G(\alpha',\beta')$ if $\beta \neq \beta'$ and $\beta,\beta'>0$. 
(Note that this is not true if $\beta=0$ or $\beta'=0$. For example, $G(2,0) \sim G(1,2)$.) 

Here, we see the following: 
\begin{lem}\label{lem.max}
Assume that $G \sim G(\alpha,\beta)$ for $\alpha \in \ZZ_{\geq 0}$ and $\beta \in \ZZ_{>0}$. 
Then we have $i(G') \leq \beta$ for any graph $G'$ with $G \sim G'$. 
\end{lem}
\begin{proof}
If $i(G')=0$, then the result is clear, so assume that $i(G')\geq 1$.
Notice that any relative mutation $\mu_{v \leftarrow w}$ keeps the isolated vertices unchanged when $v$ is not an isolated vertex. 
Thus, by Lemma~\ref{lem.rm}, there is a sequence of relative mutations which transforms $G'$ into $G(\alpha',\beta')$ with $\beta' \geq i(G')$. 
Therefore, $$G(\alpha,\beta) \sim G \sim G' \sim G(\alpha',\beta') \text{ and }\beta' \geq i(G').$$ 
Suppose that $i(G')>\beta$. Then $\beta' \geq i(G') > \beta$, a contradiction. 
\end{proof}

We recall the following useful lemma on the point scheme $\calE_{\e}$. 
\begin{lem}[{cf. \cite[Theorem 2.3]{Ukp}}]
$\calE_{\e}=\bigcap _{\e_{ij}\e_{jk}\e_{ki}=-1}\calV(x_ix_jx_k)\subset \PP^{n-1}$. 
\end{lem}

Let $\ell_\e$ denote the number of irreducible components of $\calE_\e$ that are isomorphic to $\PP^1$. 
For the investigation of $\ell_\e$, we will recall some fundamental facts on the Stanley--Reisner ideals of simplicial complexes. 
Consult, e.g., \cite[Section 5]{BH}. 

Let $\Delta$ be a simplicial complex on the vertex set $V=\{x_1,\ldots,x_n\}$, i.e., 
$\Delta \subset 2^V$ satisfies ``$F \in \Delta, F' \subset F \Longrightarrow F' \in \Delta$''. 
A maximal $F \in \Delta$ is said to be a \emph{facet} of $\Delta$. 
We define the Stanley--Reisner ideal $I_\Delta \subset k[x_1,\ldots,x_n]$ of $\Delta$ as follows: 
$$I_\Delta=\left( \prod_{x_i \in F} x_i \mid F \in 2^V, \; F \not\in \Delta\right).$$
Clearly, $I_\Delta$ is a squarefree monomial ideal.
Conversely, every squarefree monomial ideal can be realized as a Stanley--Reisner ideal $I_\Delta$ for some $\Delta$. 

For a facet $F$ of $\Delta$, let $I_F$ be the prime ideal generated by all variables $x_i$ with $x_i \not\in F$, 
i.e., $I_F=(x_i \mid x_i \in V \setminus F)$. 
It is known that \begin{align}\label{decomp}I_\Delta = \bigcap_{F \text{ : facet of }\Delta}I_F.\end{align}

For the purpose of this section, for a graph $G$ on the vertex set $V:=V(G)=\{x_1,\ldots,x_n\}$ with its edge set $E:=E(G)$, 
we consider the squarefree monomial ideal $I_G$ defined as follows: 
\begin{align*}
I_G&=(x_ix_jx_k \mid x_ix_j \not\in E \text{ and } x_ix_k \not\in E \text{ and } x_jx_k \not\in E) \\
&+ (x_ix_jx_k \mid x_ix_j \in E \text{ and }x_ix_k \in E \text{ and }x_jx_k \not\in E).
\end{align*}
In fact, we can easily see that $\calV(I_{G_\e})=\calE_\e$. 

Let $\Delta_G$ be the associated simplicial complex, i.e., $I_{\Delta_G}=I_G$. 
We write $\ell(G)$ for the number of facets of $\Delta_G$ with cardinality $2$.
By the primary decomposition \eqref{decomp}, we see that
\[\ell_\e = \ell(G_\e),\]
so we focus on the calculation of $\ell(G)$. 
By definition of the Stanley--Reisner ideal, we have the following: 
\begin{align*}
\ell(G)&=\#\{\{x_i,x_j\} \subset  V \mid \{x_i,x_j,x_k\} \not\in \Delta_G \text{ for any } x_k \in V \setminus \{x_i,x_j\}\} \\
&=\#\{\{x_i,x_j\} \subset  V \mid x_ix_jx_k \in I_G \text{ for any } x_k \in V \setminus \{x_i,x_j\}\} \\
&=\#\{\{x_i,x_j\} \subset  V \mid x_ix_j \in E, \\ 
&\quad\quad\quad
\text{``}x_ix_k \in E \text{ and } x_jx_k \not\in E\text{''} \text{ or }
\text{``}x_ix_k \not\in E \text{ and } x_jx_k \in E\text{''} \text{ for any } x_k \in V \setminus \{x_i,x_j\}\} \\
&+\#\{\{x_i,x_j\} \subset  V \mid x_ix_j \not\in E, \\
&\quad\quad\quad
\text{``}x_ix_k \in E \text{ and } x_jx_k \in E\text{''} \text{ or }
\text{``}x_ix_k \not\in E \text{ and } x_jx_k \not\in E\text{''} \text{ for any } x_k \in V \setminus \{x_i,x_j\}\}. 
\end{align*}

Assume that $G$ contains an isolated vertex $x$. In this case, $x_ix \not\in E$ and $x_jx \not\in E$ hold
for any $\{x_i,x_j\} \subset V$ with $x_ix_j \in E$.
On the other hand, the condition
\begin{center}
$x_ix_j \not\in E$ such that ``$x_ix_k \in E$ and $x_jx_k \in E$'' or ``$x_ix_k \not\in E$ and $x_jx_k \not\in E$'' 

for any $x_k \in V \setminus \{x_i,x_j\}$ 
\end{center}
is equivalent to $N_G(x_i)=N_G(x_j)$. Hence, in the case $G$ contains an isolated vertex, we see that 
\begin{align*}
\ell(G)&=\#\{\{x_i,x_j\} \subset V \mid N_G(x_i)=N_G(x_j)\}. 
\end{align*}
Notice that $N_G(u)= \emptyset$ is equivalent to what $u$ is an isolated vertex. Therefore, in the case $G$ contains an isolated vertex, we conclude that 
\begin{align*}
\ell(G)=\#\{\{x_i,x_j\} \subset V \mid N_G(x_i)=N_G(x_j) \neq \emptyset\} + \binom{i(G)}{2}. 
\end{align*}
Let $J(G)=\{\{x_i,x_j\} \subset V \mid N_G(x_i)=N_G(x_j) \neq \emptyset\}$ and let $j(G)=\#J(G)$. 

\begin{ex}\label{ex.l}
If
\[
G_\e=
\xy /r2pc/: 
{\xypolygon6{~={90}~*{\xypolynode}~>{}}},
"1";"2"**@{-},
"2";"3"**@{-},
"3";"4"**@{-},
"4";"1"**@{-}
\endxy,
\]
then $J(G_\e)=\{ \{1,3\}, \{2,4\} \}$ and $\Iso(G_\e)=\{5,6\}$, so it follows that $\ell_\e= \ell(G_\e)= 2+ \binom{2}{2}=3$.
In fact, we can verify that 
\begin{align*}
\calE_\e= &\calV(x_1, x_2, x_5) \cup \calV(x_1, x_2, x_6) \cup \calV(x_1, x_4, x_5) \cup \calV(x_1, x_4, x_6)\\
&\cup \calV(x_2, x_3, x_5)\cup \calV(x_2, x_3, x_6) \cup \calV(x_3, x_4, x_5)\cup \calV(x_3, x_4, x_6)\\
&\cup \calV(x_1, x_2, x_3, x_4) \cup \calV(x_1, x_3, x_5, x_6) \cup \calV(x_2, x_4, x_5, x_6) \;\; \subset \PP^5,
\end{align*}
and so $\ell_\e=3$.
\end{ex}

\begin{lem}\label{lem.J(G)}
Let $G$ be a graph containing at least one isolated vertex. Assume that $j(G)>0$. 
Let $\{u_1,u_2\} \in J(G)$ and let $u_2,\ldots,u_m$ be all the vertices with $N_G(u_1)=N_G(u_i)$ for $i=2,\ldots,m$. 
Let $G'=\mu_{u_m \leftarrow u_1}(\cdots\mu_{u_3 \leftarrow u_1}(\mu_{u_2 \leftarrow u_1}(G))\cdots)$. Then $j(G')< j(G)$ and $\ell(G') \geq \ell(G)$. 
\end{lem}
\begin{proof}
Note that $\{u_i, u_j\} \in J(G)$ for any $1 \leq i < j \leq m$. 
It is easy to see that $u_2,\ldots,u_m$ become isolated vertices after applying $\mu_{u_2 \leftarrow u_1},\ldots,\mu_{u_m \leftarrow u_1}$. See below. 
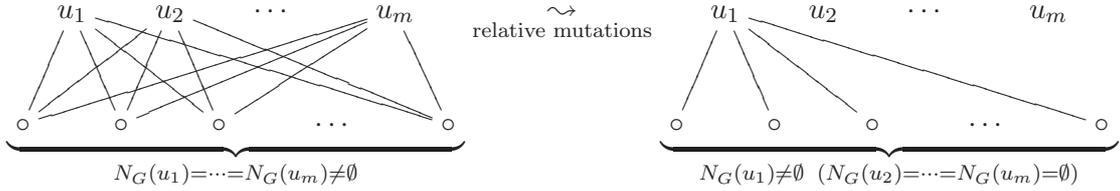
\[
\underbrace{\xymatrix@R=1pc@C=0.3pc{
&u_1 \ar@{-}[ddl]\ar@{-}[ddr]\ar@{-}[ddrrr]\ar@{-}[ddrrrrrrr]
&&u_2 \ar@{-}[ddlll]\ar@{-}[ddl]\ar@{-}[ddr]\ar@{-}[ddrrrrr]
&&\cdots
&&u_m \ar@{-}[ddlllllll]\ar@{-}[ddlllll]\ar@{-}[ddlll]\ar@{-}[ddr]\\
\\
\circ &&\circ &&\circ &&\cdots && \circ}
}_{N_G(u_1)=\cdots =N_G(u_m) \neq \emptyset}
\underset{\text{relative mutations}}{\leadsto}
\underbrace{\xymatrix@R=1pc@C=0.3pc{
&u_1 \ar@{-}[ddl]\ar@{-}[ddr]\ar@{-}[ddrrr]\ar@{-}[ddrrrrrrr]
&&u_2 
&&\cdots
&&u_m \\
\\
\circ &&\circ &&\circ &&\cdots && \circ}
}_{N_G(u_1)\neq \emptyset \;\; (N_G(u_2)=\cdots =N_G(u_m)=\emptyset)}
\]
Thus $i(G')=i(G)+m-1$. Moreover, we also see that $J(G') = J(G) \setminus \{\{u_i,u_j\} \mid 1 \leq i<j \leq m\}$. 
In fact, since only the adjacencies of $u_i$'s and the adjacencies of the vertices in $N_G(u_1)$ change after applying the above relative mutations, 
we only need to observe the vertices of $N_G(u_1)$, but we can easily see that 
$\{w,w'\} \in J(G)$ if and only if $\{w,w'\} \in J(G')$ for any $w,w' \in N_G(u_1)$. 

Therefore, we conclude that $j(G')< j(G)$ and
\begin{align*}
\ell(G') = j(G')+\binom{i(G')}{2} =j(G)-\binom{m}{2} + \binom{i(G)+m-1}{2} 
\geq j(G)+\binom{i(G)}{2} = \ell(G), 
\end{align*}
as required. (The inequality above follows from the inequality $\displaystyle \binom{a+b-1}{2}-\binom{a}{2}-\binom{b}{2} \geq 0$, 
which is true for any positive integers $a,b$.) 
\end{proof}

\begin{ex}
Work with the same graph as in Example~\ref{ex.l}. 
Take $\{1,3\} \in J(G_\e)$. Then $3$ is the only vertex $i$ with $N_{G_\e}(1)=N_{G_\e}(i)$.
Apply the relative mutation $\mu_{3 \leftarrow 1}$ to $G_\e$. Then $G_\e$ becomes 
\[
G_{\e'}=
\xy /r2pc/: 
{\xypolygon6{~={90}~*{\xypolynode}~>{}}},
"2";"1"**@{-},
"1";"4"**@{-}
\endxy. 
\]
Since $J(G_{\e'})=\{ \{2,4\} \}$, we see $j(G_{\e'})<j(G_\e)$.
Moreover, apply $\mu_{4 \leftarrow 2}$ to $G_{\e'}$. Then $G_{\e'}$ becomes $G(1,4)$.
Hence $\ell(G_\e) \leq \ell(G_{\e'}) \leq \ell(G(1,4))=\binom{4}{2}=6$. 
\end{ex}

Now let us prove Theorem~\ref{thm4}.

\begin{proof}[\textit{\textbf{Proof of Theorem \ref{thm4}}}]
By \cite[Lemma 6.5]{MU}, mutation does not change the point scheme, 
so we may assume that $G_\e$ is a graph containing an isolated vertex by Lemma~\ref{lem.mu}.
It follows from Lemma~\ref{lem.rm} that $G_\e \sim G(\alpha,\beta)$ for some $\alpha \in \ZZ_{\geq 0}$ and $\beta \in \ZZ_{>0}$.
By the proof of Theorem~\ref{thm1}, we see that $r=\beta-1$. 
Thus our goal here is to prove that $\ell_\e \leq \binom{\beta}{2}$. 
Applying  Lemma~\ref{lem.J(G)} repeatedly, 
we can obtain a graph $G'$ such that $G' \sim G_\e$, $j(G')=0$, and $\ell(G') \geq \ell(G_\e)$.
Since $i(G') \leq \beta$ by Lemma~\ref{lem.max}, we conclude by Lemma~\ref{lem.J(G)} that 
$$\ell_\e = \ell(G_\e) \leq \ell(G') = \binom{i(G')}{2} \leq \binom{\beta}{2},$$
as desired. 
\end{proof}

\section*{Acknowledgments}
The first author was supported by JSPS Grant-in-Aid for Early-Career Scientists 17K14177. 
The second author was supported by JSPS Grant-in-Aid for Early-Career Scientists 18K13381.

\end{document}